\documentclass[final,leqno]{siamltex}

\usepackage{amsmath,amssymb,amsfonts,enumerate,url}
\usepackage{nameref,hyperref,url}
\usepackage{todonotes}

\newcommand{\mc}{\mathcal}
\newcommand{\R}{\mathbb{R}}\newcommand{\Rn}{\R^n}

\renewcommand{\O}{\Omega}

\renewcommand{\l}{\lambda}
\renewcommand{\L}{\Lambda}

\newcommand{\sol}{\mathcal{S}}

\DeclareMathOperator{\cl}{\mathcal{L}}

\def\XXint#1#2#3{{\setbox0=\hbox{$#1{#2#3}{\int}$}
\vcenter{\hbox{$#2#3$}}\kern-.5\wd0}}

\title{A simple characterization of $H$-convergence for a class of nonlocal problems}
\author{Jos\'e C. Bellido\thanks{E.T.S.I. Industriales,
Department of Mathematics,
University of Castilla-La Mancha,
E--13071 Ciudad Real, Spain.
\texttt{josecarlos.bellido@uclm.es}}
\and
Anton Evgrafov\thanks{Department of Mechanical Engineering,
Technical University of Denmark,
DK--2800 Kgs. Lyngby, Denmark.
\texttt{aaev@mek.dtu.dk}}
}

\begin{document}

\maketitle

\begin{abstract}
This is a follow-up of a paper by Fern\'andez-Bonder-Ritorto-Salort \cite{bonder2017}, where the classical concept of $H$-convergence was extended to fractional \(p\)-Laplace type operators. In this short paper we provide an explicit characterization of this notion by demonstrating that the weak-\(*\) convergence of the coefficients is an equivalent condition for $H$-convergence of the sequence of nonlocal operators. This result takes advantage of nonlocality and is in stark contrast to the local \(p\)-Laplacian case.
\end{abstract}



\pagestyle{myheadings}
\thispagestyle{plain}
\markboth{J. C. Bellido and A. Evgrafov}{A sufficient condition for $H$-convergence in nonlocal problems}

\section{Introduction}

For $p\in (1,\infty)$, $s\in(0,1)$ and a \textit{nonlocal} conductivity $a(x,y)$ belonging to the class
\[\mc{A}_{\lambda,\Lambda}=\left\{ a\in L^\infty(\Rn\times \Rn)\;:\; a(x,y)=a(y,x),\, \lambda\le a(x,y)\le \Lambda, \, \text{a.e.\ in \(\Rn\times \Rn\)}\right\},\]
where $0<\l\le \L<\infty$ are given constants, let us consider the following nonlocal operator related to the fractional \(p\)-Laplacian:
\[\cl_a u(x)=p.v. \int_{\Rn}a(x,y)\frac{|u(x)-u(y)|^{p-2}(u(x)-u(y))}{|x-y|^{n+sp}}\,dy.\]
For a fixed bounded domain $\O\subset \Rn$ with Lipschitz boundary and $f\in L^{p'}(\O)$, with
$p'=(p-1)/p$ being the conjugate exponent to $p$, we consider the nonlocal problem
\begin{equation}\label{eq:vp}
\left\{ \begin{array}{ll} \cl_{a} u=f & \mbox{in }\O,\\
u=0& \mbox{in }\R^n\backslash \O.  \end{array}\right.
\end{equation}
This problem is well-posed, with a unique solution found in the space
\[ W_0^{s,p}(\O)=\left\{u\in W^{s,p}(\Rn)\;:\; u=0 \mbox{ a.e. in }\Rn\backslash \O\right\},\]
where $W^{s,p}(\R^n)$ is the classical fractional Sobolev space over \(\Rn\), see~\cite{DiPaVa12,AdFo03}:
\[W^{s,p}(\Rn)=\left\{ u\in L^p(\Rn)\;:\; D_{s,p}u(x,y) \in L^p(\Rn\times \Rn)\right\},\]
and
\begin{equation}\label{eq:nlocgrad}D_{s,p}u(x,y)= \frac{u(x)-u(y)}{|x-y|^{\frac{n}{p}+s}}\end{equation}
is the $(s,p)$-nonlocal gradient of $u$.
In view of existence and uniqueness of solutions we can employ the shorthand notation \((u,q) = \sol_{a}f\) to denote the solution \(u \in W^{s,p}_0(\O)\) to~\eqref{eq:vp} corresponding
to coefficients \(a\) and the right hand side \(f\), and the \emph{non-local flux}
\(q=a|D_{s,p}u|^{p-2} D_{s,p} u \in L^{p'}(\O\times\O)\).

In this paper we are concerned with \(H\)-convergence of the nonlocal operators $\cl_{a_k}$ for a given sequence of coefficients $a_k$.
\begin{definition}\label{de:1}
  Given the sequence of coefficients $\{ a_k \}_{k=1}^\infty \subset \mc{A}_{\l,\L}$,
  we say that $\cl_{a_k}$ $H$-converges to $\cl_a$ if for any \(f\in W^{-s,p}(\O)\) (the dual space of $W^{s,p}_0(\O)$)
  the following conditions are satisfied:
\begin{enumerate}
\item convergence of states: $u_k \rightharpoonup u$, weakly in $W_0^{s,p}(\O)$;
\item convergence of non-local fluxes: $q_k\rightharpoonup q$, weakly in $L^{p'}(\O\times\O)$,
\end{enumerate}
where \((u_k,q_k)=\sol_{a_k}f\), and \((u,q)=\sol_{a}f\).
\end{definition}

In~\cite{bonder2017} is was shown that for any sequence $\{a_k\}_{k=1}^\infty \subset\mc{A}_{\lambda,\Lambda}$ there exists a subsequence $\{a_{k'}\}$ and a function $a\in \mc{A}_{\l,\frac{\L^{p'}}{\l}}$ such that $\cl_{a_{k'}}$ $H$-converges to $\cl_{a}$.
We show that in fact a \(a\) belongs to the same class
\(\mc{A}_{\lambda,\L}\).
Furthermore, our main result establishes that weak-\(*\) convergence of the sequence of coefficients is a necessary and sufficient condition for $H$-convergence in the considered case.

\begin{theorem}\label{th:sufficient} If $a_k, a\in \mc{A}_{\lambda,\Lambda}$, then $a_k\rightharpoonup a$ weakly-\(*\) in $L^\infty(\Rn\times \Rn)$ if and only if $\cl_{a_k}$ $H$-converges to $\cl_{a}$.
\end{theorem}

This result generalizes, giving a simpler proof, previous results for a related nonlocal situation in the linear case~\cite[Th. 6]{andres2015nonlocal}. We also refer to~\cite{waurick}, where an abstract, general setting for nonlocal $H$-convergence is analyzed.

The outline of the paper is as follows. Section~2 is devoted to setting the functional analysis framework of this work, and to stating the main results from~\cite{bonder2017}, which are the starting point of the investigation presented here. Section~3 deals with the relation of the nonlocal $H$-convergence notion introduced in Definition~\ref{de:1} with the weaker notion of $G$-convergence. We show that the two notions are equivalent.  Additionally, we establish the uniqueness of $H$-limit. Finally, Section~4 contains the proof of Theorem~\ref{th:sufficient}.

\section{Preliminaries}

In this section we set the functional analysis framework in which problems are set and recall the main results from~\cite{bonder2017}.

We start by recalling some fundamental facts about fractional Sobolev spaces. The space $W^{s,p}(\Rn)$, previously defined,  is equipped with the norm
\[ \|u\|_{s,p}=\left(\|u\|_p^p+|u|_{s,p}^p\right)^{\frac{1}{p}},\]
where $\|u\|_p$ is the usual norm of $u$ in $L^p(\Rn)$ and
\[ |u|_{s,p}=\left(\int\int_{\Rn\times \Rn}\frac{|u(x)-u(y)|^p}{|x-y|^{n+sp}}\,dx\,dy\right)^\frac{1}{p}=\left(\int\int_{\Rn\times \Rn}\left|D_{s,p} u(x,y)\right|^p\,dx\,dy\right)^\frac{1}{p}\]
is the Gagliardo seminorm \cite{DiPaVa12,AdFo03}. With this definition $W^{s,p}(\Rn)$ is a separable and reflexive Banach space for $1<p<\infty$. $W^{s,p}_0(\O)$ is usually defined as
\[W_0^{s,p}(\O)=\overline{C_c^\infty(\O)}^{\|\cdot\|_{s,p}},\]
and in the case $\O$ has a Lipschitz boundary the following identification holds
\[ W_0^{s,p}(\O)=\left\{u\in W^{s,p}(\Rn)\;:\; u=0 \mbox{ a.e. in }\Rn\backslash \O\right\}.\]
An important mathematical fact is that, for $\O$ bounded, $W^{s,p}_0(\O)$ embeds continuously into $L^p(\O)$, thanks to the Poincar\'e's inquality in this fractional situation: there exists $C=C(n,s,|\O|)>0$ such that
\[\|u\|_p\le C|u|_{s,p},\]
for all $u\in W_0^{s,p}(\O)$. Furthermore, the Rellich-Kondrachov theorem can be extended to fractional Sobolev spaces, and the embedding of $W^{s,p}_0(\O)$ into $L^p(\O)$ is compact. Proofs of these results can be found, for instance, in \cite{DiPaVa12}.
The dual space of $W^{s,p}_0(\O)$ is denoted by $W^{-s,p'}(\O)$, and its norm is given by
\[\|f\|_{-s,p'}=\sup\left\{ \langle f,u\rangle \;:\; u\in W_0^{s,p}(\O), \, |u|_{s,p}=1\right\}.\]

We now focus on the precise statement of the problem~\eqref{eq:vp}, which should be understood in the weak sense. Thus, we require that $\cl_au=f$ holds in the sense of distributions, and we say that $u \in W_0^{s,p}(\O)$ is a solution to~\eqref{eq:vp} if
\begin{equation}\label{eq:wf}
  \frac{1}{2}\int\int_{\Rn\times \Rn}a(x,y) \frac{|u(x)-u(y)|^{p-2}(u(x)-u(y))(v(x)-v(y))}{|x-y|^{n+sp}}\,dx\,dy=\langle f, v\rangle, \end{equation}
for all test functions $v\in C_0^\infty(\O)$. The first result documents the well-posedness of the problem~\eqref{eq:vp} \cite[Prop. 2.2, Cor. 2.4]{bonder2017}.

\begin{proposition}\label{prop:existence} For each $a\in \mc{A}_{\l,\L}$ and $f\in W^{-p,s}(\O)$, the problem~\eqref{eq:vp} admits a unique solution $u \in W_0^{s,p}(\O)$.  Furthermore, this solution is also the unique minimizer in $W_0^{s,p}(\O)$ of functional
\[I_a(v)=\frac{1}{2p}\int\int_{\Rn\times \Rn}a(x,y)|D_{s,p}v(x,y)|^p\,dx\,dy-\langle f,v\rangle.\]
\end{proposition}

In addition to the non-local gradient~\eqref{eq:nlocgrad}, for each $\phi\in L^{p'}(\Rn\times \Rn)$ it will be convenient to define the corresponding $(s,p)$-divergence operator by
\[d_{s,p} \phi(x)=p.v. \int_{\Rn} \frac{\phi(x,y)-\phi(y,x)}{|x-y|^{\frac{n}{p}+s}}\,dy.\]
The following result summarizes several properties of these operators \cite[Th. 3.1., Lem. 3.3]{bonder2017}.
\begin{theorem} \label{th:H-convergence}The following assertions hold:
\begin{enumerate}
\item \label{integration by parts} \textbf{Integration by parts:}
For each $\phi\in L^{p'}(\Rn\times \Rn)$ and $u\in W^{s,p}(\Rn)$ we have the inclusion
$d_{sp} \phi \in W^{-s,p'}(\Rn)$, the dual of $W^{s,p}(\Rn)$, and the integration by parts formula
\[\int\int_{\Rn\times \Rn} \phi D_{s,p}u \,dx\,dy=\langle d_{s,p}\phi,u\rangle;\]
\item \label{div-curl} \label{nonlocal div-curl lemma} \textbf{Nonlocal div-curl lemma:} given $\phi_k,\phi\in L^{p'}(\Rn\times \Rn)$ and $v_k,\,v\in W^{s,p}(\Rn)$, \(k=1,2,\dots\) such that $v_k\rightharpoonup v$, weakly in $W^{s,p}(\Rn)$, $\phi_k\rightharpoonup \phi$, weakly in $L^{p'}(\Rn\times \Rn)$ and $d_{s,p} \phi_k\rightarrow d_{s,p} \phi$, strongly in $W_{loc}^{-s,p'}(\Rn)$ then
\[ \phi_kD_{s,p}v_k \rightarrow \phi D_{s,p}v \]
in the sense of distributions.
\end{enumerate}
\end{theorem}

Note that owing to the integration by parts identity and the assumed symmetry of the conductivity
\(a(x,y)=a(y,x)\),  equation~\eqref{eq:wf} can be equivalently understood as
\(d_{s,p} q = 2f\), where \((u,q)=\sol_{a}f\).

There are several other results in the literature, which are related to the previous one. In other works dealing with nonlocal or fractional problems a nonlocal vector calculus has been developed in order deal with the involved operators. References including integration by parts formulas are~\cite{du2013, mengesha2015,mengesha2017,bellido2018}. Regarding the div-curl lemma, this is a very interesting compensated compactness-type result in the nonlocal context. In~\cite{waurick2} a general analytic perspective for div-curl lemma that includes the nonlocal situation is considered. It is interesting to refer to~\cite{bellido2018}, where in a very related situation to the one analyzed here, the weak convergence of any minor of the Riesz fractional gradient of vector fields has been shown by means of a nonlocal Piola identity.

Now we are prepared to state our point of departure, \cite[Theorem~4.6]{bonder2017}, which establishes that $\mc{A}_{\l,\L}$ is sequentially relatively compact with respect to $H$-convergence.

\begin{theorem} \label{th:compactness} $0<\l\le\L$. For any sequence $\{a_k\}\subset \mc{A}_{\l,\L}$, there exists a subsequence $\{a_{k'}\}$ and $a\in \mc{A}_{\l,\L}$ such that $\cl_{a_k}$ $H$-converges to $\cl_a$.
\end{theorem}

To be precise this is not the exact statement of~\cite[Theorem~4.6]{bonder2017}, as it differs in the upper bound on the coefficients of the $H$-limiting problem. In ~\cite[Theorem~4.6]{bonder2017} it is claimed that the $H$-limit $a\in \mc{A}_{\l,\frac{\L^{p'}}{\l}}$.  However, we will show that the more natural upper bound
\begin{equation}\label{eq:upperbound}
a(x,y)\le\L,\quad \text{a.e.\ in \(\Rn\times\Rn\)},
\end{equation}
holds in this case.
Indeed, let us assume that $\cl_{a_k}$ $H$-converges to $a$. Let us further fix an arbitrary $f\in L^{p'}\setminus\{0\}$, and put $(u_k,q_k) = \sol_{a_k}f$, and \((u,q)=\sol_{a}f\). Since $a_k\in \mc{A}_{\l,\L}$, for any $\varphi\in C_c^\infty(\Rn\times \Rn)$, \(\varphi\geq 0\) we have
\[\begin{aligned}
 \int\int_{\Rn\times \Rn}
 |q_k|^{p'} \varphi\,dx\,dy
 &= \int\int_{\Rn\times \Rn}a_k^{p'}|D_{s,p} u_k|^{p} \varphi\,dx\,dy\\
 &\le \L^{p'-1}  \int\int_{\Rn\times \Rn}\varphi q_k D_{s,p} u_k \,dx\,dy.
 \end{aligned}\]
 The term on the left is weakly lower semicontinuous with respect to the fluxes, which
 converge weakly owing to the \(H\)-convergence assumption.
 The term on the right converges owing to the non-local div-curl lemma.
 Passing to the limit we therefore arrive at the inequality
\[\int\int_{\Rn\times \Rn}a^{p'}|D_{s,p} u|^{p} \varphi\,dx\,dy\le \L^{p'-1}  \int\int_{\Rn\times \Rn}a|D_{s,p} u|^{p} \varphi\,dx\,dy,\]
and as $\varphi$ is nonnegative but otherwise arbitrary,
\begin{equation}\label{eq:bound}
a^{p'}|D_{s,p} u|^{p}\le \L^{p'-1}a|D_{s,p} u|^{p},\quad\text{a.e.\ in \(\Rn\times\Rn\)}.
\end{equation}
Additionally, since \(f\) is arbitrary, $u$ is also arbitrary, and \eqref{eq:bound} holds for any $u\in W^{s,p}_0(\O)$, and hence the inequality~\eqref{eq:upperbound} holds.

\section{$G$-convergence}

In the local case, $H$-convergence was proposed by Murat and Tartar, see for example~\cite{Murat2}, as an extension of the previously proposed $G$-convergence concept~\cite{spagnolo1967sul}. $G$-convergence was formulated for linear elliptic equations with symmetric coefficients in divergence form and only requires weak convergence of the states. $H$-convergence, on the other hand, requires convergence of both states and fluxes, and has been formulated for problems with non-symmetric coefficients. In the case of elliptic PDEs with symmetric coefficients both notion are known to coincide.  In the more general case of non-symmetric coefficients $G$-convergence is less useful in the sense that the $G$-limit is not guaranteed to be unique~\cite[Section~1.3.2]{allaire2012shape}.

The previous definition of nonlocal $H$-convergence requires both convergence of the states and convergence of the fluxes. As a consequence of this, owing to the div-curl lemma, the associated energy
\[E(a)= \int\int_{\Rn\times \Rn} a |D_{s,p} u|^p\,dx\,dy = \int\int_{\Rn\times \Rn} q D_{s,p} u\,dx\,dy, \]
with \((u,q)=\sol_{a}f\) is continuous with respect to $H$-convergence. This is a remarkable and a desirable property, especially when dealing with optimal design problems. In this section,  we show that in the considered nonlocal situation of scalar and symmetric coefficients in $\mc{A}_{\l,\L}$, the requirement on flux convergence in the definition of $H$-convergence is unnecessary. In other words, nonlocal $G$-convergence implies $H$-convergence, precisely as in the local case.  First of all we make rigorous the definition of nonlocal $G$-convergence.

\begin{definition}
  Given the sequence of coefficients $\{ a_k \}_{k=1}^\infty \subset \mc{A}_{\l,\L}$,
  we say that $\cl_{a_k}$ $G$-converges to $\cl_a$ if for any \(f\in W^{-s,p}(\O)\)
  we have
  \[u_k\rightharpoonup u,\quad \mbox{weakly in } W^{s,p}(\O),\]
  where \((u_k,q_k)=\sol_{a_k}f\) and \((u,q)=\sol_{a}f\).
\end{definition}

\begin{proposition} \label{prop:GimpliesH}
Consider a sequence $\{a_k\}_{k=0}^\infty \subset \mc{A}_{\l,\L}$.  Then  $\cl_{a_k}$ $H$-converges to $\cl_{a}$
if and only if $\cl_{a_k}$ $G$-converges to $\cl_{a}$, with \(a\in \mc{A}_{\l,\L}\).
\end{proposition}
\begin{proof} Obviously $H$-convergence implies $G$-convergence.
  For the sake of contradiction, let us now assume that we have \(G\)-convergence, but not \(H\)-convergence.
  For an arbitrary $f\in W^{-s,p}(\O)$ let $(u^f_k,q^f_k)=\sol_{a_k}f$ and $(u^f,q^f)=\sol_{a}f$.
  Owing to the assumed \(G\)-convergence we have \(u^f_{k} \rightharpoonup u^f\) in \(W^{s,p}_0(\O)\).
  The assumed lack of \(H\)-convergence is equivalent to saying that for some \(\hat{f}\in W^{-s,p}(\O)\)
  there is a subsequence \(a_{k'}\) and a \(L^{p'}(\Rn\times\Rn)\)-weakly open neighbourhood \(\hat{N}\) of
  \(q^{\hat{f}}\), such that \(q^{\hat{f}}_{k'}\not\in \hat{N}\), for all \(k'\).
  Owing to Theorem~\ref{th:compactness}, the set \(\{\cl_{a_{k'}}\}\) is relatively sequentially compact with respect to \(H\)-convergence. Therefore there is a further subsequence $a_{k''}$ and $\tilde{a}\in \mc{A}_{\l,\L}$ such that $\cl_{a_{k''}}$ $H$-converges to $\cl_{\tilde{a}}$. Consequently, if we put $(\tilde{u}^f,\tilde{q}^f)=\sol_{\tilde{a}}f$,
then \(u^f=\tilde{u}^f\) owing to the assumed \(G\)-convergence and the uniqueness of the weak limit, and the uniqueness of solutions to~\eqref{eq:wf}.
Therefore we have the equality
\[\int\int_{\Rn\times \Rn} (a-\tilde{a})|D_{s,p}u^f|^p\,dx\,dy=0,\]
which holds for an arbitrary \(f\).
Since $f$ is arbitrary, $u^f\in W^{s,p}_0(\O)$ is also arbitrary. Applying \cite[Proposition~17]{elbau}, which characterizes null nonlocal functionals, we obtain that necessarily $a=\tilde{a}$, a.e.\ in $\Rn\times \Rn$.
But then \(\tilde{q}^f=q^f\), \(q_{k''}^f\rightharpoonup q^f\) in \(L^{p'}(\Rn\times\Rn)\), and consequently
\(q^{\hat{f}}_{k''}\in \hat{N}\) for all large \(k''\), which is a contradiction.
\end{proof}

Another important point is the uniqueness of the $G$-limit. This is established in the next result.
\begin{proposition}\label{prop:uniqueness}
The $G$-limit of a sequence in $\mc{A}_{\l,\L}$ is unique.
\end{proposition}
\begin{proof} The proof follows the lines of the second part of proof of Proposition \ref{prop:GimpliesH}. Let us assume that the sequence $\{a_n\}\subset \mc{A}_{\l,\L}$ $G$-converges to both $a$ and $\tilde a$. Then, arguing as in the proof of Proposition \ref{prop:GimpliesH}, we have that
\[\int\int_{\Rn\times \Rn} (a-\tilde{a})|D_{s,p}u^f|^p\,dx\,dy=0,\]
for any \(f\), where $u^f$ is the solution of the Dirichlet problem \eqref{eq:vp} for both coefficients $a$ and $ \tilde a$. The same argument as above yields the conclusion
\[a=\tilde{a}, \mbox{  a.e. in  }\Rn\times \Rn.\]
\end{proof}

\section{Proof of Theorem~\ref{th:sufficient}}

This section is devoted to the proof Theorem~\ref{th:sufficient}.

\paragraph{We claim that weak-* convergence is sufficient for $H$-convergence}
Indeed, let us assume that \(\{a_k\}_{k=1}^\infty \in \mc{A}_{\l,\L}\) be such that
\(a_k \rightharpoonup a \in \mc{A}_{\l,\L}\), weak-* in \(L^\infty(\Rn\times\Rn)\).
Owing to Theorem~\ref{th:compactness} there exists a subsequence $\{a_{k'}\}$ and $\tilde{a}\in \mc{A}_{\l,\L}$ such that $\cl_{a_{k'}}$ $H$-converges to $\cl_{\tilde{a}}$. Let $f\in W^{-s,p}(\O)$ be fixed but arbitrary, and let $(u_k,q_k)=\sol_{a_k}f$, \((u,q)=\sol_{a}f\), and \((\tilde{u},\tilde{q})=\sol_{\tilde{a}}f\).
Owing to $H$-convergence, $u_{k'}\rightharpoonup \tilde{u}$, weakly in $W^{s,p}_0(\O)$.
Recalling that $u_{k'}=0$ in $\Rn\backslash \O$ and the compact embedding of $W^{s,p}_0(\O)$ into $L^p(\O)$, $\|u_{k'}-\tilde{u}\|_{L^p(\Rn)}\to 0$.
In particular, there is a further subsequence \(u_{k''}\) of \(u_{k'}\), such that
\(u_{k''}(x)\to \tilde{u}(x)\), for almost all \(x \in \Rn\).

Owing to the variational characterization of solutions to~\eqref{eq:wf} given in Proposition~\ref{prop:existence}, we have the inequality \(I_{a_k}(u_k)\le I_{a_k}(u)\), \(\forall\: k=1,2,\dots\)
Taking into account the facts that $|D_{s,p}u|^p\in L^1(\Rn\times \Rn)$ and $a_k\rightharpoonup a$ weak-\(*\) in $L^\infty(\Rn\times \Rn)$, we obtain the inequality
\begin{equation}\label{eq:limsup}
  \limsup_{k''\rightarrow \infty} I_{a_{k''}}(u_{k''})\le \lim_{k''\rightarrow \infty} I_{a_{k''}}(u)=I_{a}(u)\leq I_{a}(\tilde{u}).\end{equation}

On the other hand, let us define the measures
\[\nu_k(E) =\int\int_E a_k(x,y)\,dx\,dy =\int\int_{\Rn\times \Rn} \chi_E(x,y)a_k(x,y)\,dx\,dy,\quad k\ge 1,\]
and
\[\nu(E)=\int\int_E a(x,y)\,dx\,dy =\int\int_{\Rn\times \Rn} \chi_E(x,y)a(x,y)\,dx\,dy,\]
where $E\subset \Rn\times \Rn$ is an arbitrary Lebesgue measurable set. Weak-\(*\) convergence of $a_k$ to $a$ implies the strong convergence of these measures, that is, $\lim_{k\rightarrow \infty}\nu_k(E)=\nu(E)$ for any measurable set $E\subset \R^n\times \Rn$. Since $u_{k''}(x)\to \tilde{u}(x)$ for almost all $x\in \Rn$, it follows that $|D_{s,p}u_{k''}(x,y)|^p\rightarrow |D_{s,p} \tilde{u}(x,y)|^p$, for almost all $(x,y)\in \Rn\times \Rn$.
These facts together with the upper bound~\eqref{eq:limsup} allow us to apply the generalized Fatou's lemma~\cite[Proposition~17, p.~269]{royden} to get the inequality
\begin{equation*}\begin{aligned}
I_{a}(\tilde{u}) &=
\frac{1}{2p}\int\int_{\Rn\times\Rn} |D_{s,p}\tilde{u}(x,y)|^p\,d\nu(x,y) - \langle f,u\rangle \\ &\leq
\liminf_{k''\rightarrow \infty} \frac{1}{2p}\int\int_{\Rn\times\Rn} |D_{s,p}u_{k''}(x,y)|^p\,d\nu_{k''}(x,y) - \langle f,u_{k''}\rangle
=\liminf_{k''\rightarrow \infty} I_{a_{k''}}(u_{k''}).
\end{aligned}\end{equation*}
Therefore \(I_{a}(u)=I_{a}(\tilde{u})\), whence \(u=\tilde{u}\), owing to the uniqueness of solutions to~\eqref{eq:wf} and their variational characterization.
Arguing further as in the proof of Proposition~\ref{prop:GimpliesH}, we conclude that $a=\tilde{a}$, almost everywhere in \(\Rn\times\Rn\).
Finally, since from every subsequence of \(\cl_{a_k}\) we can extract a further subsequence, which \(H\)-converges to  \(\cl_{a}\), the whole sequence must converge to \(\cl_{a}\).

\paragraph{We now claim that weak-* convergence is also necessary for $H$-convergence}
Assume that $\cl_{a_k}$ $H$-converges to $\cl_{a}$, but for some weak-* open
neighbourhood \(N\) of \(a \in \mc{A}_{\l,\L}\) and a subsequence \(k'\) we have
\(a_{k'} \not\in N\).
Since  $\{a_{k'}\}_{k'=1}^\infty \subset \mc{A}_{\l,\L}$ and is thus bounded in \(L^\infty(\Rn\times\Rn)\), it has a non-empty set of weak-* limit points.
Suppose that \(a_{k''} \rightharpoonup \tilde{a}\in \mc{A}_{\l,\L}\) for some further
subsequence \(k'' = 1,2,\dots\)
By the already established implication, \(\cl_{a_{k''}}\) \(H\)-converges to \(\cl_{\tilde{a}}\).
Owing to Propositions~\ref{prop:GimpliesH}, \ref{prop:uniqueness}, and~\cite[Proposition~17]{elbau}
we necessarily have \(\tilde{a}=a\).
But then \(a_{k''} \in N\), for all large enough \(k''\), which is a contradiction.
This completes the proof.

\subsection*{Acknowledgements}
AE's research is financially supported by the Villum Fonden through the Villum Investigator Project InnoTop.  The work of JCB is funded by FEDER EU and Ministerio de Econom{\'i}a y Competitividad (Spain) through grant MTM2017-83740-P.

%\bibliographystyle{siam}
%\bibliography{biblio}

\end{document}